\documentclass[11pt]{article}
\usepackage[T1]{fontenc}
\usepackage{libertinus}

\usepackage{microtype}
\usepackage[a4paper,margin=30mm]{geometry}
\usepackage{amsmath,amssymb,amsthm,mathtools}
\usepackage[hidelinks]{hyperref}
\usepackage{enumitem}
\usepackage{comment}

\newtheorem{theorem}{Theorem}[section]
\newtheorem{proposition}[theorem]{Proposition}
\newtheorem{lemma}[theorem]{Lemma}
\newtheorem{corollary}[theorem]{Corollary}
\newtheorem{remark}[theorem]{Remark}

\newcommand{\R}{\mathbb R}
\newcommand{\E}{\mathbb E}
\newcommand{\Prob}{\mathbb P}
\newcommand{\conv}{\operatorname{conv}}
\newcommand{\vol}{\operatorname{vol}}
\newcommand{\Cov}{\operatorname{Cov}}

\newcommand{\1}{\mathbf 1}
\newcommand{\cH}{\mathcal H}
\newcommand{\ip}[2]{\left\langle #1,#2\right\rangle}
\newcommand{\norm}[1]{\left\lVert #1\right\rVert}

\title{Many Facets in Random Polytopes from Product and Log-Concave Measures}
\author{
  Silouanos Brazitikos\thanks{University of Crete, Greece. Email: \texttt{silouanb@uoc.gr}}
  \and
Minas Pafis \thanks{University of Athens, Greece. Email: \texttt{mipafis@math.uoa.gr}
}}
\date{}

\begin{document}
\maketitle

\begin{abstract}\footnotesize
We prove bounds of order $n^{n/2}e^{O(n)}$ for the expected number of
facets of high-dimensional random polytopes. First, let $\mu$ be a
non-degenerate compactly supported even probability measure on $\R$ satisfying
$\mu([x^\ast-s,x^\ast])\asymp s^\kappa$ near its right endpoint $x^\ast$.
For every sufficiently small fixed $\alpha>0$, the convex hull of
$N=\lfloor e^{\alpha n}\rfloor$ independent points with law $\mu^{\otimes n}$
has at least $n^{n/2}e^{-C_{\mu,\alpha}n}$ expected facets; this includes all
symmetric finite-alphabet distributions.

For every full-dimensional log-concave probability measure on $\R^n$, we
prove that there exist $T\in[n,2n]$ and
$N=\lceil e^Tn^{3/2}\rceil$ for which
\[
n^{n/2}e^{-Cn}
\leq
\mathbb E f_{n-1}(P_N)
\leq
n^{n/2}e^{Cn}.
\]
Thus the scale $n^{n/2}$, up to exponential factors, is universal for
log-concave measures in this high-dimensional exponential regime.
Finally, we construct a symmetric isotropic full-support non-log-concave
counterexample with only $(1+o(1))2^n$ expected facets.
\end{abstract}
\begin{comment}
We prove two lower bounds for the number of facets of random polytopes in high
 dimension.  The first concerns product measures.  Let $\mu$ be a non-degenerate
 even probability measure on the real line with compact support and assume that
 its mass near the right endpoint $x^*$ satisfies
\[
 c s^\kappa\leq \mu([x^*-s,x^*])\leq C s^\kappa
\]
for all sufficiently small $s>0$, where $\kappa\geq0$.  For every sufficiently
small fixed $\alpha>0$, if $N=\lfloor e^{\alpha n}\rfloor$ and
$X_1,\ldots,X_N$ are independent with law $\mu^{\otimes n}$, then
\[
 \E f_{n-1}\bigl(\conv\{X_1,\ldots,X_N\}\bigr)
 \geq n^{n/2}e^{-C_{\mu,\alpha}n}.
\]
In particular this applies to every symmetric finite-alphabet distribution;
for the symmetric Bernoulli law it gives a $0/1$ polytope with at least
$(cn)^{n/2}$ facets.  We also exhibit a symmetric isotropic full-support
non-log-concave measure for which $N=e^{\alpha n+o(n)}$ random points have only
$(1+o(1))2^n$ expected facets.

The second result has no product assumption.  For every full-dimensional
log-concave probability measure $\mu$ on $\R^n$, there are
$T\in[n,2n]$ and $N=\lceil e^T n^{3/2}\rceil$ such that the convex hull of
$N$ independent $\mu$-distributed points has expected number of facets at
least $n^{n/2}e^{-Cn}$.  The proof uses the Cram\'er transform, the recent
solution of the slicing problem, a lower bound for the mean volume of a random
simplex from a log-concave density, and the Blaschke--Petkantschin formula.
\end{comment}

\section{Introduction}

For an $n$-dimensional polytope $P$ we write $f_{n-1}(P)$ for its number of
facets.  A classical extremal problem asks for the largest possible value
\[
 g(n)=\max\{f_{n-1}(P): P\subset\R^n\text{ is a }0/1\text{ polytope}\}.
\]
The problem, associated with Fukuda and Ziegler and recorded in
\cite{Ziegler}, has a long history.  Fleiner, Kaibel and Rote proved the upper
bound $g(n)\leq 30(n-2)!$ for all sufficiently large $n$ \cite{FKR}.  On the
lower-bound side, B\'ar\'any and P\'or obtained the first superexponential
estimate
\[
 g(n)\geq \left(\frac{cn}{\log n}\right)^{n/4}
\]
in \cite{BaranyPor}.  Gatzouras, Giannopoulos and Markoulakis subsequently
proved
\[
 g(n)\geq \left(\frac{cn}{(\log n)^2}\right)^{n/2}
\]
in \cite{GGM2005}, and then improved this to
\[
 g(n)\geq \left(\frac{cn}{\log n}\right)^{n/2}
\]
in \cite{GGM2007}.  Very recently Friedland obtained the log-free estimate
\[
 g(n)\geq (cn)^{n/2}
\]
by a direct argument for random sign polytopes \cite{Friedland2026}.

The probabilistic study of $0/1$ polytopes is closely connected with threshold
phenomena for random convex hulls.  An early source of the method is the work of
Dyer, F\"uredi and McDiarmid on random points in the discrete cube
\cite{DFM}.  In a more general setting, Pafis studied random polytopes generated
by product measures and expressed the relevant thresholds through the
one-dimensional Cram\'er transform \cite{Pafis2024}.  For log-concave measures,
connections between the Cram\'er transform, half-space depth and thresholds for
random polytopes were developed in \cite{BGPHalfspace,BGP2024,BC2025,GT2026}.  The atomic
case and some limitations of discrete log-concavity were investigated in
\cite{BP2026}.  We also use the general convex-hull estimates of Hayakawa,
Lyons and Oberhauser \cite{HLO}.

The purpose of this paper is to place the many-facets phenomenon in two broad
probabilistic settings.  The first one is genuinely non-log-concave and relies
on product structure.  Let $\mu$ be a compactly supported even probability
measure on $\R$, and let
\[
 x^*=\sup\{x>0:\mu([x,\infty))>0\}.
\]
We assume that for some $\kappa\geq0$,
\begin{equation}\label{eq:intro-tail}
 c s^\kappa\leq \mu([x^*-s,x^*])\leq C s^\kappa,
 \qquad (0<s\leq s_0).
\end{equation}
This includes measures with an atom at the endpoint ($\kappa=0$), and hence
all non-degenerate symmetric finitely supported distributions.  It also
includes many purely atomic measures with infinitely many atoms accumulating
at the endpoints.  Under \eqref{eq:intro-tail}, for every sufficiently small
fixed $\alpha>0$ and $N=\lfloor e^{\alpha n}\rfloor$, we prove
\[
 \E f_{n-1}\bigl(\conv\{X_1,\ldots,X_N\}\bigr)
 \geq n^{n/2}e^{-C_{\mu,\alpha}n},
\]
where the $X_i$ are independent with law $\mu^{\otimes n}$.  Notice that we do
not symmetrize the sample.  For the Bernoulli law the resulting polytope is
already a sign polytope and is affinely equivalent to a $0/1$ polytope.

Before turning to the second theorem, we give a simple full-support symmetric
isotropic example showing that without product structure or log-concavity the
expected number of facets may remain only exponential in $n$.  

The second
theorem concerns every full-dimensional log-concave probability measure on
$\R^n$, with no product or symmetry assumption.  The log-concave result is in fact two-sided.  We prove that for every
full-dimensional log-concave probability measure on $\mathbb R^n$ and every
$N\geq2n$,
\[
\mathbb E f_{n-1}(P_N)
\leq
C^n N n^{(n-3)/2}.
\]
Consequently, throughout every fixed exponential window
$N\leq e^{bn}n^{3/2}$ one has
\[
\mathbb E f_{n-1}(P_N)
\leq
C_b^n n^{n/2}.
\]
Combined with a similar lower bound, this shows that for a suitable
$N=e^{\Theta(n)}n^{3/2}$ the expected number of facets of a random polytope
generated by an arbitrary log-concave measure is
\[
n^{n/2}e^{O(n)}.
\]
The point is that this conclusion is uniform over the entire class of
log-concave measures, despite the very different facet asymptotics exhibited
by particular models when the dimension is fixed and $N\to\infty$.
The number of sample points is larger than in the product theorem; the product
structure is what permits one to work at Cram\'er levels $B_t:=\{\Lambda^*_{\mu}<t\}$ of size $\alpha n$
with arbitrarily small fixed $\alpha$.

It is worth contrasting the preceding result with the classical theory of
random polytopes.  A substantial part of that theory concerns random points
chosen from a fixed convex body in a fixed dimension, with the number of
points tending to infinity.  In this setting the asymptotic behaviour of the
face numbers depends strongly on the geometry of the underlying body.  For
smooth convex bodies it is governed by boundary curvature and affine surface
area, whereas for polytopal containers a quite different, logarithmic
behaviour occurs; see, for instance,
\cite{BaranyLarman1988,Barany1992,Reitzner2005}.  Exact formulae are also
available for several highly symmetric models, including Gaussian and beta
polytopes; see, e.g.,
\cite{HugMunsoniusReitzner2004,KabluchkoTemesvariThale2019}.  Thus there is no
universal dependence on the sample size throughout the class of log-concave
measures when the dimension is kept fixed and the number of points tends to
infinity.

The situation is rather different in the genuinely high-dimensional
exponential regime considered here.  For Gaussian polytopes,
B\"or\"oczky, Lugosi and Reitzner
\cite{BoroczkyLugosiReitzner2024} proved, in the regime \(N/n\to\infty\),
the asymptotic formula
\[
\E f_{n-1}(P_N)
=
\left((4\pi+o(1))\log\frac{N}{n}\right)^{(n-1)/2}.
\]
In particular, when \(\log N\) is proportional to \(n\), the expected number
of facets is \(n^{n/2}\) up to factors exponential in \(n\).  Bonnet and
O'Reilly obtained a closely related high-dimensional picture for spherical
random polytopes, including the exponential regime
\(\log N\asymp n\) \cite{BonnetOReilly2022}.  The general log-concave result
proved here shows that this \(n^{n/2}\) scale is not a consequence of Gaussian
or spherical symmetry.  Up to exponential factors, it persists for every
full-dimensional log-concave probability measure, without any assumption of
smoothness, compact support, symmetry, product structure, or an explicit form
of the density.  In this sense, the result is of a different nature from the
classical fixed-dimensional stochastic approximation results: it is a
high-dimensional universality statement for the number of facets in the
window \(N=e^{\Theta(n)}\).  Outside this window one should not expect such a
universal behaviour, as the classical smooth, polytopal, and Gaussian models
already exhibit substantially different asymptotics.

We briefly describe the two proofs.  In the product case, the endpoint
assumption gives uniform estimates for moments under one-dimensional
exponential tilts.  Combined with a coarse half-space-depth estimate from
\cite{Pafis2024} and Berry--Esseen, this yields the sharp local estimate
\[
 q_{\mu^{\otimes n}}(x)\gtrsim
 \frac{e^{-\Lambda_{\mu^{\otimes n}}^*(x)}}{\sqrt n}
\]
on a fixed cube and on Cram\'er levels comparable with $n$, together with the
matching upper estimate for the tangent half-space.  A large portion of a
Cram\'er level surface lies inside a smaller fixed cube.  For every point of
this portion, a finite set of at most $D_L^n$ points at a slightly lower level
meets every half-space that penetrates sufficiently far into the Cram\'er body.
The finite set is obtained from the VC $\varepsilon$-net theorem
\cite[Corollary~3.7]{HausslerWelzl}.  The depth estimate and \cite[Proposition~13]{HLO} show
that all these points belong to the random polytope with overwhelmingly high
probability.  Finally, strong convexity shows that an exterior half-space
which misses a lower Cram\'er body intersects the chosen level surface in a
set of surface area at most $C_L^n|S^{n-1}|$.  Comparing the total exposed area
with the contribution of one facet gives the result.

The log-concave proof for the lower bound is shorter.  After isotropic normalization, the solution
of the slicing problem \cite{KlartagLehec} gives the required uniform density
bound for the original measure and all relevant exponential tilts.  We select
a Cram\'er level through the volume of the corresponding set of tilt
parameters.  On a substantial set of directions the variance of the tilted
one-dimensional statistic is between $a^2$ and $4a^2$, where
$c\sqrt n\leq a\leq Cn$.  The density of a hyperplane section and the
conditional log-concave density on that section can then be controlled
directly.  A simple random-simplex lemma converts the latter estimate into a
lower bound for the weighted section-simplex integral.  The
Blaschke--Petkantschin formula \cite{SchneiderWeil} then gives the desired
facet count.

For the upper bound, we make again use of the Blaschke--Petkantschin formula. Corresponging one-dimensional variance bounds for the tilted measures allow us to control the probability of certain half-spaces. Combining again with the random-simplex lemma yields the proof.

\section{Compactly supported product measures}

Throughout this section $\mu$ is a non-degenerate even Borel probability
measure on $\R$ with support contained in $[-x^*,x^*]$, where $x^*>0$, and we
assume \eqref{eq:intro-tail}.  Let $X$ have distribution $\mu$ and for \(t\in\R\) write
\[
\Lambda_\mu(t)=\log\E e^{tX}.
\]
Define the tilted probability measure
\[
 dP_t=e^{tX-\Lambda_\mu(t)}\,d\mu,
\]
and denote expectation under $P_t$ by $\E_t$.  Then
\[
 \E_tX=\Lambda_\mu'(t),\qquad \operatorname{Var}_t(X)=\Lambda_\mu''(t).
\] In addition, if $t_1,\ldots,t_n \in \mathbb{R}$ we also define the probability measure \[P_{t_1,\ldots,t_n}=P_{t_1}\otimes\cdots \otimes P_{t_n}.\]
The function $\Lambda_\mu$ is even and strictly convex, and
$\Lambda_\mu':\R\to(-x^*,x^*)$ is a strictly increasing bijection.  Its
Legendre transform $\Lambda_\mu^*$ is even, strictly convex and smooth on
$(-x^*,x^*)$, and
\begin{equation}\label{eq:phi-derivatives}
 (\Lambda_\mu^*)'(x)=(\Lambda_\mu')^{-1}(x),
 \qquad
 (\Lambda_\mu^*)''(x)=\frac{1}{\Lambda_\mu''((\Lambda_\mu^*)'(x))}.
\end{equation}
Since $\Lambda_\mu''(t)=\operatorname{Var}_t(X)\leq(x^*)^2$, we have the global
strong-convexity estimate
\begin{equation}\label{eq:phi-strong}
 (\Lambda_\mu^*)''(x)\geq (x^*)^{-2},
 \qquad |x|<x^*.
\end{equation}
For $\mu_n=\mu^{\otimes n}$ and $r>0$ we write $
 B_r=\{x\in[-x^*,x^*]^n:\Lambda_{\mu_n}^*(x)\leq r\}$ and an important property in the product case is
\[
 \Lambda_{\mu_n}^*(x)=\sum_{i=1}^n\Lambda_\mu^*(x_i).
\]
We also write
\[
 q_n(x)=\inf\{\mu_n(H):H\text{ is a closed half-space containing }x\}.
\]
for Tukey's half-space depth of $\mu_n$.
\subsection{One-dimensional estimates}

The endpoint assumption is used only through the following elementary
consequences.

\begin{lemma}\label{lem:compact-tilts}
For every fixed integer $m\geq1$,
\begin{equation}\label{eq:uniform-tilted-moment}
 \sup_{t\in\R}|t|^m\E_t|X-\Lambda_\mu'(t)|^m<\infty.
\end{equation}
Moreover,
\begin{equation}\label{eq:J-ratio}
 \frac{t^2\Lambda_\mu''(t)}
 {\Lambda_\mu^*(\Lambda_\mu'(t))}\longrightarrow2
 \qquad (t\to0).
\end{equation}
Finally,
\begin{equation}\label{eq:tail-cramer}
 \frac{-\log\mu([x,\infty))}{\Lambda_\mu^*(x)}\longrightarrow1
 \qquad (x\uparrow x^*).
\end{equation}
\end{lemma}

\begin{proof}
By symmetry it is enough to consider $t>0$.  Put $Y=x^*-X$ and let $\nu$ be
its law. Notice that \[|X-\mathbb{E}_t X|=|Y-\mathbb{E}_tY|.\] So a standard probabilistic argument implies that it suffices to prove that for every $m \in \mathbb{N}$ \[\sup_{t>0}t^m \mathbb{E}_t Y^m<+\infty.\]Let $t>0$. Observe that
\begin{equation}\label{eq:tilted-Y-ratio}
 t^m\E_tY^m
 =
 \frac{\displaystyle\int_0^{2x^{\ast}}(ty)^me^{-ty}\,d\nu(y)}
 {\displaystyle\int_0^{2x^{\ast}}e^{-ty}\,d\nu(y)}.
\end{equation}
If $t\leq s_0^{-1}$, then
\[
 t^m\E_tY^m\leq\left(\frac{2x^{\ast}}{s_0}\right)^m.
\]
We may therefore assume that $t\geq s_0^{-1}$. For the denominator
of \eqref{eq:tilted-Y-ratio}, we obtain
\begin{align}
 \int_0^{2x^{\ast}}e^{-ty}\,d\nu(y)
 &\geq\int_0^{1/t}e^{-ty}\,d\nu(y)\notag\\
 &\geq e^{-1}\nu([0,1/t])\notag\\
 &=e^{-1}\mu([x^{\ast}-1/t,x^{\ast}])
 \geq ce^{-1}t^{-\kappa}.
 \label{eq:denominator-lower}
\end{align}
For the numerator, write
\begin{equation}\label{eq:numerator-split}
 \int_0^{2x^{\ast}}(ty)^me^{-ty}\,d\nu(y)
 =
 \int_{[0,s_0]}(ty)^me^{-ty}\,d\nu(y)
 +\int_{(s_0,2x^{\ast}]}(ty)^me^{-ty}\,d\nu(y)
 =:I_1+I_2.
\end{equation}
Clearly,
\begin{equation}\label{eq:I2}
 I_2\leq(2tx^{\ast})^me^{-ts_0}.
\end{equation}
For
\[
 J_k=\left[\frac{k}{t},\frac{k+1}{t}\right),
 \qquad k=0,1,\ldots,\lfloor ts_0\rfloor-1,
\]
condition \eqref{eq:intro-tail} gives
\begin{align}
 \int_{J_k}(ty)^me^{-ty}\,d\nu(y)
 &\leq(k+1)^me^{-k}\nu([0,(k+1)/t))\notag\\
 &\leq Ct^{-\kappa}(k+1)^{\kappa+m}e^{-k}.
 \label{eq:Jk}
\end{align}
The remaining interval satisfies the same estimate:
\begin{align}
 \int_{[\lfloor ts_0\rfloor/t,s_0]}(ty)^me^{-ty}\,d\nu(y)
 &\leq C t^{-\kappa}
 (\lfloor ts_0\rfloor+1)^{\kappa+m}e^{-\lfloor ts_0\rfloor}.
 \label{eq:last-interval}
\end{align}
Consequently,
\begin{equation}\label{eq:I1}
 I_1\leq C t^{-\kappa}
 \sum_{k=0}^{\lfloor ts_0\rfloor}
 (k+1)^{\kappa+m}e^{-k}
 \leq C_{\kappa,m}t^{-\kappa}.
\end{equation}
Combining \eqref{eq:tilted-Y-ratio},
\eqref{eq:denominator-lower}, \eqref{eq:I2} and \eqref{eq:I1}, we get
\[
 t^m\E_tY^m
 \leq C_{\kappa,m}+C_m t^{\kappa+m}e^{-ts_0},
\]
which is uniformly bounded for $t\geq s_0^{-1}$, and thus \eqref{eq:uniform-tilted-moment} is proved.

Since
\[
 \Lambda_\mu^*(\Lambda_\mu'(t))
 =t\Lambda_\mu'(t)-\Lambda_\mu(t)
 =\int_0^t s\Lambda_\mu''(s)\,ds
 =t^2\int_0^1u\Lambda_\mu''(tu)\,du,
\]
continuity of $\Lambda_\mu''$ at the origin gives \eqref{eq:J-ratio}.
For \eqref{eq:tail-cramer} we refer to \cite[Theorem~1.3]{GatzourasGiannopoulos2009}.
\end{proof}
Let
\[
 r^*=\sup_{|x|<x^*}\Lambda_\mu^*(x)\in(0,\infty].
\]
By \eqref{eq:tail-cramer}, the product-measure depth estimate of
\cite[Theorem~3.3]{Pafis2024} applies: for every $\zeta>0$, if $n$ is sufficiently large,
then
\begin{equation}\label{eq:coarse-depth-product}
 \inf_{y\in B_r}q_n(y)
 \geq\exp\bigl(-(1+\zeta)r-2\zeta n\bigr),
 \qquad r>0.
\end{equation}
When $\mu$ has endpoint atoms, contact points on the boundary of
$[-x^*,x^*]^n$ are interpreted through the normal-cone argument in
\cite[Theorem~4.9]{BP2026}.

\subsection{A local depth estimate}

\begin{theorem}\label{thm:product-local-depth}
There exists
\[
 0<\eta<\min\{1,x^*/4\}
\]
with the following property.  For every fixed $0<a<b<r^*$ there are constants
$c,C>0$ such that, for all sufficiently large $n$ and every
$x\in[-\eta,\eta]^n$ satisfying
\[
 an\leq\Lambda_{\mu_n}^*(x)\leq bn,
\]
one has
\begin{equation}\label{eq:local-depth-product}
 q_n(x)\geq c\frac{e^{-\Lambda_{\mu_n}^*(x)}}{\sqrt n}.
\end{equation}
Moreover, if
\[
 H_x^+=\{y:\ip{\nabla\Lambda_{\mu_n}^*(x)}{y-x}\geq0\},
\]
then
\begin{equation}\label{eq:tangent-upper-product}
 \mu_n(H_x^+)\leq C\frac{e^{-\Lambda_{\mu_n}^*(x)}}{\sqrt n}.
\end{equation}
\end{theorem}

\begin{proof}
Fix a sufficiently small absolute number $\varepsilon>0$.  By
\eqref{eq:J-ratio}, continuity of the tilted second and third centered moments
at $0$, and \eqref{eq:phi-derivatives}, we may choose $\eta>0$ so small that,
whenever $|z|\leq2\eta$ and $t=(\Lambda_\mu^*)'(z)$,
\begin{equation}\label{eq:local-variance-product}
 \Lambda_\mu^*(z)\leq t^2\Lambda_\mu''(t)\leq3\Lambda_\mu^*(z)
\end{equation}
and
\begin{equation}\label{eq:local-third-product}
 |t|^3\E_t|X-z|^3
 \leq \varepsilon\,t^2\Lambda_\mu''(t).
\end{equation}
The choice of $\eta$ is independent of $a$ and $b$.

Fix $0<a<b<r^*$ and a closed half-space $H$ containing $x$.  If $0\in H$,
evenness gives $\mu_n(H)\geq1/2$.  Assume therefore that $0\notin H$.  Choose
$z\in H\cap[-x^*,x^*]^n$ minimizing $\Lambda_{\mu_n}^*$ and put
\[
 r_H=\Lambda_{\mu_n}^*(z)\leq\Lambda_{\mu_n}^*(x).
\]
Then $z\in\partial H$ and $\partial H$ supports $B_{r_H}$ at $z$.

Suppose first that $z\in\partial[-x^*,x^*]^n$.  This is possible only when
$\mu$ has an endpoint atom.  The normal-cone description in \cite[Lemma~4.8]{BP2026}
then shows, after sign changes and permutation of coordinates, that the
half-space contains the event on which a certain set of coordinates is fixed
at the endpoint.  Hence
\[
 \mu_n(H)\geq e^{-r_H}\geq e^{-\Lambda_{\mu_n}^*(x)},
\]
which is stronger than the required estimate.

We may therefore assume that $z\in(-x^*,x^*)^n$.  Then
\[
 H=\{y:\ip{t}{y-z}\geq0\},
 \qquad
 t=(t_1,\ldots,t_n)=\nabla\Lambda_{\mu_n}^*(z),
\]
and $z_i=\Lambda_\mu'(t_i)$.  Put
\[
 \Delta=\Lambda_{\mu_n}^*(x)-r_H.
\]
Choose $0<\delta<a/2$ sufficiently small, and then set
$\zeta=\delta/(b+3)$.  If $\Delta\geq\delta n$, the coarse estimate
\eqref{eq:coarse-depth-product} gives
\[
 \begin{aligned}
 \mu_n(H)
 &\geq q_n(z)\\
 &\geq\exp\bigl(-(1+\zeta)r_H-2\zeta n\bigr)\\
 &\geq\exp\bigl(-\Lambda_{\mu_n}^*(x)+\zeta n\bigr),
 \end{aligned}
\]
and we are done.

It remains to consider $\Delta<\delta n$.  By the strong convexity
\eqref{eq:phi-strong} and the fact that $x\in H$,
\[
 \norm{x-z}_2^2\leq C\delta n.
\]
Let
\[
 I=\{i:|z_i|\leq2\eta\}.
\]
Since $|x_i|\leq\eta$, every $i\notin I$ satisfies
$|x_i-z_i|>\eta$, and hence
\begin{equation}\label{eq:I-complement-product}
 |I^c|\leq C_\eta\delta n.
\end{equation}

We claim that
\begin{equation}\label{eq:local-mass-product}
 \sum_{i\in I}\Lambda_\mu^*(z_i)\geq c_{\eta,a}n.
\end{equation}
Indeed, $t_iz_i\geq0$ and, since $x\in H$,
\[
 \ip{t}{z}\leq\ip{t}{x}\leq\eta\sum_{i=1}^n|t_i|.
\]
For $i\notin I$ we have $|z_i|>2\eta$, so
\[
 \sum_{i\in I^c}|t_i|\leq\sum_{i\in I}|t_i|.
\]
Moreover, $\Lambda_\mu\geq0$ implies
\[
 \Lambda_\mu^*(z_i)=t_iz_i-\Lambda_\mu(t_i)\leq x^*|t_i|.
\]
On the fixed interval $[-2\eta,2\eta]$, the relation
$\Lambda_\mu^*(\Lambda_\mu'(t))\asymp_\eta t^2$ that follows by \eqref{eq:J-ratio} and the positivity of $\Lambda_\mu''$, gives
\[
 \sum_{i\in I}|t_i|
 \leq C_\eta\sqrt{n\sum_{i\in I}\Lambda_\mu^*(z_i)}.
\]
Since $r_H\geq(a-\delta)n\geq an/2$, we obtain
\[
 \frac a2n
 \leq \sum_{i\in I}\Lambda_\mu^*(z_i)
     +C_\eta x^\ast\sqrt{n\sum_{i\in I}\Lambda_\mu^*(z_i)},
\]
which proves \eqref{eq:local-mass-product}.

Under the product tilt $P_{t_1,\ldots,t_n}$ set
\[
 S=\sum_{i=1}^nt_i(X_i-z_i),
\]
and write
\[
 \sigma^2=\sum_{i=1}^nt_i^2\Lambda_\mu''(t_i),
 \qquad
 \rho=\sum_{i=1}^n|t_i|^3\E_{t_i}|X_i-z_i|^3.
\]
The local estimates \eqref{eq:local-variance-product},
\eqref{eq:local-third-product}, the lower bound
\eqref{eq:local-mass-product}, the global bounds
\eqref{eq:uniform-tilted-moment} for $m=2,3$, and
\eqref{eq:I-complement-product} give
\[
 c_{\eta,a}n\leq\sigma^2\leq C_\mu n
\]
and
\[
 \frac{\rho}{\sigma^2}\leq\varepsilon+C_{\eta,a}\delta.
\]
After first fixing $\eta$ and then choosing $\delta$ sufficiently small, the
right-hand side is as small as desired.

The change-of-measure identity is
\begin{equation}\label{eq:change-measure-product}
 \mu_n(H)
 =e^{-r_H}\E_{t_1,\ldots,t_n}
 \bigl[e^{-S}\1_{\{S\geq0\}}\bigr].
\end{equation}
Let $G$ be a centered Gaussian random variable with variance $\sigma^2$.
Berry--Esseen gives
\[
 \sup_{s \in \mathbb{R}}\left|
 P_{t_1,\ldots,t_n}(S\leq s)-\Prob(G\leq s)
 \right|
 \leq C\frac{\rho}{\sigma^3}.
\]
The function $s\mapsto e^{-s}\1_{[0,\infty)}(s)$ has bounded variation, and
the standard Mills estimate gives
\[
 \E[e^{-G}\1_{\{G\geq0\}}]\asymp\frac1\sigma.
\]
It follows that
\[
 \E_{t_1,\ldots,t_n}
 \bigl[e^{-S}\1_{\{S\geq0\}}\bigr]
 \geq\frac c{\sqrt n}.
\]
Since $r_H\leq\Lambda_{\mu_n}^*(x)$, \eqref{eq:change-measure-product} proves
\eqref{eq:local-depth-product}.

For the tangent half-space at $x$, take $z=x$.  Then every coordinate lies in
$[-\eta,\eta]$, so \eqref{eq:local-variance-product} gives
\[
 an\leq\sigma^2\leq3bn,
\]
and \eqref{eq:local-third-product} gives $\rho\leq\varepsilon\sigma^2$.
The same Berry--Esseen argument, using the upper Mills estimate, gives
\[
 \E[e^{-S}\1_{\{S\geq0\}}]\leq\frac C{\sqrt n},
\]
which proves \eqref{eq:tangent-upper-product}.
\end{proof}

\subsection{Three geometric estimates}

From now on $\eta$ is fixed as in Theorem~\ref{thm:product-local-depth}.  We use
only the smaller cube
\[
 Q=[-\eta/2,\eta/2]^n,
\]
while the larger cube will always be written explicitly as
$[-\eta,\eta]^n$.  This fixed gap is the only separation that is needed.

\begin{lemma}\label{lem:product-level-area}
There are $\alpha_0,C>0$, depending only on $\mu$, such that for every
$n\geq2$ and every $0<t\leq\alpha_0n$,
\[
 \cH^{n-1}(\partial B_t\cap Q)
 \geq e^{-Cn}t^{(n-1)/2}|S^{n-1}|.
\]
\end{lemma}

\begin{proof}
Choose $0<m\leq M<\infty$ so that
\begin{equation}\label{eq:star-normal-local-equivalence}
   \frac m2s^2\leq\Lambda_\mu^*(s)\leq\frac M2s^2,
 \qquad |s|\leq\eta/2.
\end{equation}
Let
\[
 \Omega_n=\left\{u\in S^{n-1}:\norm{u}_\infty\leq\frac2{\sqrt n}\right\}.
\]
There is an absolute constant $c_0>0$ such that
$|\Omega_n|\geq c_0^n|S^{n-1}|$.  Choose $\alpha_0>0$ so small that
\[
 2\sqrt{\frac{2\alpha_0}{m}}\leq\frac\eta2.
\]
Let $u\in\Omega_n$ and $t\leq\alpha_0n$. Then  \eqref{eq:star-normal-local-equivalence} implies that the ray $\{ru:r\geq0\}$ meets
$\partial B_t$ at a point $r_t(u)u \in Q $, with
\[
 \sqrt{\frac{2t}{M}}\leq r_t(u)\leq\sqrt{\frac{2t}{m}}.
\] The radial surface element is at least $r_t(u)^{n-1}du$.  Hence
\[
 \cH^{n-1}(\partial B_t\cap Q)
 \geq c_0^n\left(\frac{2t}{M}\right)^{(n-1)/2}|S^{n-1}|,
\]
as required.
\end{proof}

\begin{lemma}\label{lem:product-finite-set}
Fix $0<a<b<r^*$.  There exists $L_0>0$ such that, for every fixed
$L\geq L_0$, there are constants $C_L,D_L>1$ with the following property.
For all sufficiently large $n$, if
\[
 an\leq t\leq bn,
 \qquad
 x\in\partial B_t\cap Q,
\]
then there is a finite set $W_x$ with $|W_x|\leq D_L^n$ satisfying
\begin{equation}\label{eq:Wx-location}
 W_x\subset
 \left\{w\in[-\eta,\eta]^n:
 t-C_L\sqrt n\leq\Lambda_{\mu_n}^*(w)\leq t-\frac L4
 \right\},
\end{equation}
and every open half-space $G$ such that
\begin{equation}\label{eq:halfspace-penetrates}
 x\in G,
 \qquad
 G\cap B_{t-2L}\cap[-\eta,\eta]^n\neq\varnothing
\end{equation}
meets $W_x$.  Moreover,
\[
 \log D_L=O_{\mu,\eta}(\log(1+L)).
\]
\end{lemma}

\begin{proof}
Put
\[
 m=\min_{|s|\leq\eta}(\Lambda_\mu^*)''(s)>0,
 \quad
 M=\max_{|s|\leq\eta}(\Lambda_\mu^*)''(s)<\infty,
 \quad
 A=\max_{|s|\leq\eta}|(\Lambda_\mu^*)'(s)|<\infty.
\]
Fix $G$ satisfying \eqref{eq:halfspace-penetrates}, and let $z$ minimize
$\Lambda_{\mu_n}^*$ on $\overline G\cap[-\eta,\eta]^n$.  Then
$\Lambda_{\mu_n}^*(z)\leq t-2L$.  Since $x\in Q$, the segment $[z,x]$ lies in
$[-\eta,\eta]^n$.  Choose $y\in[z,x]$ with
\[
 \Lambda_{\mu_n}^*(y)=t-L.
\]
The function $s\mapsto\Lambda_{\mu_n}^*((1-s)z+sx)$ is non-decreasing, so
$\ip{\nabla\Lambda_{\mu_n}^*(y)}{x-y}\geq0$.  Strong convexity gives
\[
 \norm{x-y}_2\leq\sqrt{\frac{2L}{m}}.
\]
Set $c=(x+y)/2$.  Then $c\in G$,
\[
 c\in[-3\eta/4,3\eta/4]^n,
 \qquad
 \Lambda_{\mu_n}^*(c)\leq t-\frac L2,
\]
and by integrating
\[
 \Lambda_{\mu_n}^*(c)\geq t-A\sqrt{\frac{L}{2m}}\sqrt n.
\]
Since $t\geq an$, this is at least $an/2$ for all sufficiently large $n$.
The function $((\Lambda_\mu^*)')^2/\Lambda_\mu^*$ has a positive continuous extension at
$0$, hence
\[
 c_1\sqrt n\leq\norm{\nabla\Lambda_{\mu_n}^*(c)}_2\leq A\sqrt n.
\]
Write
\[
 g=\norm{\nabla\Lambda_{\mu_n}^*(c)}_2,
 \qquad
 e=\frac{\nabla\Lambda_{\mu_n}^*(c)}g,
 \qquad
 r_0=\frac\eta{32},
 \qquad
 \rho=\frac{L}{16g}.
\]
For fixed $L$ and large $n$, we have $\rho\leq r_0$.  Consider
\[
 E_G=c+\left\{u+se:u\perp e,\,
 \frac{\norm{u}_2^2}{r_0^2}+\frac{s^2}{\rho^2}\leq1\right\}.
\]
The margin between $c$ and the boundary of $[-\eta,\eta]^n$ shows that
$E_G\subset[-\eta,\eta]^n$.  Taylor's formula yields
\[
 \Lambda_{\mu_n}^*(c+u+se)
 \leq t-\frac L2+\frac L{16}
       +\frac M2(r_0^2+\rho^2)
 \leq t-\frac L4
\]
provided $L\geq L_0(\mu,\eta)$.  Also
\[
 E_G\subset B(x,R_L),
 \qquad
 R_L=\sqrt{\frac{L}{2m}}+2r_0.
\]
For $w\in E_G$, the gradient bound on $[-\eta,\eta]^n$ gives
\[
 \Lambda_{\mu_n}^*(w)\geq t-AR_L\sqrt n.
\]
Define, independently of $G$,
\[
 D_x=\left\{w\in[-\eta,\eta]^n:
 \norm{w-x}_2\leq R_L,
 \quad
 t-AR_L\sqrt n\leq\Lambda_{\mu_n}^*(w)\leq t-L/4
 \right\}.
\]
Then $E_G\subset D_x$.  Since $E_G$ is centrally symmetric about $c\in G$,
at least half of it lies in $G$.  If $\omega_n=\vol_n(B_2^n)$, then
\[
 \vol_n(E_G)=\omega_n r_0^{n-1}\rho,
 \qquad
 \vol_n(D_x)\leq\omega_nR_L^n.
\]
Since $g\leq A\sqrt n$,
\[
 \frac{\vol_n(D_x\cap G)}{\vol_n(D_x)}
 \geq \frac12\left(\frac{r_0}{R_L}\right)^{n-1}
       \frac{L}{16AR_L\sqrt n}
 \geq B_L^{-n},
\]
where $B_L=O_{\mu,\eta}(\sqrt{1+L})$.  Apply the VC
$\varepsilon$-net theorem for half-spaces \cite{HausslerWelzl} to normalized
Lebesgue measure on $D_x$, with $\varepsilon=B_L^{-n}$.  It gives a set
$W_x\subset D_x$ meeting every such half-space and satisfying
\[
 |W_x|\leq C(n+1)B_L^n\log(B_L^n)\leq D_L^n.
\]
This proves the lemma.
\end{proof}

\begin{lemma}\label{lem:product-exterior}
Fix $L>0$.  For all sufficiently large $n$, let $t>2L$ and let $H$ be a
closed half-space satisfying
\[
 \operatorname{int}H\cap
 \bigl(B_{t-2L}\cap[-\eta,\eta]^n\bigr)=\varnothing.
\]
Then
\[
 \cH^{n-1}(\partial B_t\cap Q\cap H)
 \leq\left(\frac{4L}{m}\right)^{(n-1)/2}|S^{n-1}|,
\]
where
\[
 m=\min_{|s|\leq\eta}(\Lambda_\mu^*)''(s)>0.
\]
\end{lemma}

\begin{proof}
Enlarging $H$ if necessary, assume that its boundary supports
$B_{t-2L}\cap[-\eta,\eta]^n$ at a point $z$.  Thus
\[
 H=\{y:\ip{v}{y-z}\geq0\}
\]
for some $v\neq0$.  Since the origin belongs to the interior of both sets, the
normal-cone sum rule gives
\[
 v=\lambda\nabla\Lambda_{\mu_n}^*(z)+w,
 \qquad
 \lambda\geq0,
 \quad
 w\in N_{[-\eta,\eta]^n}(z).
\]
Let $y\in Q\cap H$.  If $w\neq0$, the fixed gap between the two cubes and the coordinate description of the normal cone of a cube, give
\[
 \ip{w}{y-z}\leq-\frac\eta2\norm{w}_1<0.
\]
Hence $\lambda>0$ and
$\ip{\nabla\Lambda_{\mu_n}^*(z)}{y-z}\geq0$.  The same conclusion is immediate if
$w=0$.  Since $\lambda>0$, we also have $\Lambda_{\mu_n}^*(z)=t-2L$.  Therefore, if
$y\in B_t\cap Q\cap H$, strong convexity gives
\[
 t\geq\Lambda_{\mu_n}^*(y)
 \geq t-2L+\frac m2\norm{y-z}_2^2.
\]
Thus
\[
 B_t\cap Q\cap H
 \subset z+2\sqrt{\frac Lm}\,B_2^n.
\]
By monotonicity of surface area under inclusion for convex bodies,
\[
 \cH^{n-1}(\partial B_t\cap Q\cap H)
 \leq\left(\frac{4L}{m}\right)^{(n-1)/2}|S^{n-1}|.
\]
\end{proof}

\subsection{The theorem for product measures}

\begin{theorem}\label{thm:product-many-facets}
There exists $\alpha_0=\alpha_0(\mu)>0$ with the following property.  For
every fixed $0<\alpha<\alpha_0$, set
\[
 N=\lfloor e^{\alpha n}\rfloor,
\]
let $X_1,\ldots,X_N$ be independent random vectors with law $\mu_n$, and set
\[
 P_N=\conv\{X_1,\ldots,X_N\}.
\]
Then, for all sufficiently large $n$,
\[
 \E f_{n-1}(P_N)
 \geq n^{n/2}e^{-C_{\mu,\alpha}n}
 \geq(c_{\mu,\alpha}n)^{n/2}.
\]
\end{theorem}

\begin{proof}
Choose $\alpha_0>0$ so small that $2\alpha_0<r^*$ and
Lemma~\ref{lem:product-level-area} applies for $t\leq2\alpha_0n$.  Fix
$0<\alpha<\alpha_0$ and set
\begin{equation}\label{eq:N-t-product}
 N=\lfloor e^{\alpha n}\rfloor,
 \qquad
 t=\log N-\frac32\log n.
\end{equation}
For all sufficiently large $n$,
\[
 \frac\alpha2n\leq t\leq\alpha n.
\]
Let
\[
 \Sigma_t=\partial B_t\cap Q.
\]
By Lemma~\ref{lem:product-level-area},
\begin{equation}\label{eq:Sigma-area-product}
 \cH^{n-1}(\Sigma_t)
 \geq e^{-C_1n}t^{(n-1)/2}|S^{n-1}|.
\end{equation}

Fix $x\in\Sigma_t$.  Theorem~\ref{thm:product-local-depth} gives
\[
 p_x:=\mu_n(H_x^+)
 \leq C\frac{e^{-t}}{\sqrt n}
 =C\frac nN.
\]
If none of $X_1,\ldots,X_N$ lies in $H_x^+$, then $x\notin P_N$.  Hence
\begin{equation}\label{eq:missing-product}
 \Prob(x\notin P_N)
 \geq(1-p_x)^N
 \geq e^{-C_2n}.
\end{equation}

Fix $L\geq L_0$ as in Lemma~\ref{lem:product-finite-set}.  If $w\in W_x$,
then \eqref{eq:Wx-location} and the choice of $t$ place $w$, for all large $n$,
in the range of Theorem~\ref{thm:product-local-depth}.  Thus
\[
 q_n(w)
 \geq c\frac{e^{-\Lambda_{\mu_n}^*(w)}}{\sqrt n}
 \geq ce^{L/4}\frac{e^{-t}}{\sqrt n}
 =ce^{L/4}\frac nN.
\]
Also $\Lambda_{\mu_n}^*(w)>0$, which implies that $w \neq 0$;
by symmetry $q_n(w)\leq1/2$.  Put
\[
 q=q_n(w),
 \qquad
 r=\frac{Nq}{n}.
\]
For $L$ sufficiently large, $r\geq B:=ce^{L/4}>2$.  Proposition~13 of
\cite{HLO} yields
\[
 \Prob(w\notin P_N)
 \leq\left[r\exp\bigl(A_q(1+q-r)\bigr)\right]^n,
 \qquad
 A_q=\frac1q\log\frac1{1-q}.
\]
Since $A_q\geq1,\,q\leq1/2$ and the function $s \mapsto se^{2-s}$ is decreasing for $s \geq 1$,
\[
 \Prob(w\notin P_N)\leq(Be^{2-B})^n.
\]
Taking the union over $W_x$ and using
$\log D_L=O(\log(1+L))$, we may choose $L$ so large that
\begin{equation}\label{eq:Wx-in-product}
 \Prob(W_x\not\subset P_N)\leq e^{-(C_2+3)n}.
\end{equation}

We also need full dimensionality.  Set
\[
 p_\mu=\sup_{a\in\R}\mu(\{a\})<1.
\]
Every proper affine hyperplane $H\subset\R^n$ satisfies
$\mu_n(H)\leq p_\mu\,$: condition on all but one coordinate corresponding to a
non-zero coefficient in the equation defining $H$.  Therefore, until the affine
span of the sampled points has dimension $n$, the next point increases its
dimension with conditional probability at least $1-p_\mu$.  A binomial
lower-tail estimate gives
\begin{equation}\label{eq:full-dimensional-product}
 \Prob(P_N\text{ is not full-dimensional})\leq e^{-c_\mu N}.
\end{equation}

For $x\in\Sigma_t$, let $E_x$ be the event that $x\notin P_N$, that $P_N$ is
full-dimensional, and that every open half-space containing $x$ and meeting
$B_{t-2L}\cap[-\eta,\eta]^n$ also meets $P_N$.  By
\eqref{eq:missing-product}, \eqref{eq:Wx-in-product},
\eqref{eq:full-dimensional-product}, and Lemma~\ref{lem:product-finite-set},
\begin{equation}\label{eq:Ex-product}
 \Prob(E_x)\geq\frac12e^{-C_2n}
\end{equation}
for all sufficiently large $n$.  The event is measurable; failure of the last condition is equivalent to
\[
 \max_{u\in S^{n-1}}
 \left(
 \min\left\{\ip{u}{x},
 h_{B_{t-2L}\cap[-\eta,\eta]^n}(u)\right\}
 -h_{P_N}(u)
 \right)>0.
\]
Tonelli's theorem and \eqref{eq:Sigma-area-product} therefore give
\begin{equation}\label{eq:expected-area-product}
 \E\,\cH^{n-1}\{x\in\Sigma_t:E_x\}
 \geq e^{-C_3n}t^{(n-1)/2}|S^{n-1}|.
\end{equation}

Let $F$ be a facet of a full-dimensional realization of $P_N$, and let $H_F$
be its closed exterior half-space.  If $E_x$ occurs and $x$ strictly violates
the facet inequality of $F$, then
$x\in\operatorname{int}H_F$.  Since
$\operatorname{int}H_F\cap P_N=\varnothing$, the defining property of $E_x$
implies
\[
 \operatorname{int}H_F\cap
 \bigl(B_{t-2L}\cap[-\eta,\eta]^n\bigr)=\varnothing.
\]
Lemma~\ref{lem:product-exterior} gives
\[
 \cH^{n-1}(\Sigma_t\cap H_F)
 \leq C_L^n|S^{n-1}|.
\]
Every point counted on the left-hand side of
\eqref{eq:expected-area-product} lies outside $P_N$ and therefore strictly
violates at least one facet inequality.  Hence
\[
 \cH^{n-1}\{x\in\Sigma_t:E_x\}
 \leq f_{n-1}(P_N)C_L^n|S^{n-1}|.
\]
Taking expectations and using \eqref{eq:expected-area-product},
\[
 \E f_{n-1}(P_N)
 \geq e^{-C_4n}t^{(n-1)/2}.
\]
Since $t\geq\alpha n/2$, this is
\[
 \E f_{n-1}(P_N)
 \geq n^{n/2}e^{-C_{\mu,\alpha}n},
\]
after changing the constant in the exponential.
\end{proof}

\begin{corollary}[The $0/1$ case]\label{cor:01}
There is an absolute constant $c>0$ such that, for every sufficiently large
$n$, there exists a $0/1$ polytope $P\subset\R^n$ with
\[
 f_{n-1}(P)\geq(cn)^{n/2}.
\]
\end{corollary}

\begin{proof}
Take $\mu=\frac12(\delta_{-1}+\delta_1)$.  The endpoint condition holds with
$\kappa=0$.  Theorem~\ref{thm:product-many-facets} gives a realization of
$P_N=\conv\{X_1,\ldots,X_N\}$ with at least $(cn)^{n/2}$ facets.  The affine
map
\[
 x\longmapsto\frac{x+1}{2}
\]
sends $\{-1,1\}^n$ onto $\{0,1\}^n$ and preserves the number of facets.
\end{proof}

\begin{corollary}[Mass at the endpoint]\label{cor:finite-discrete}
Theorem~\ref{thm:product-many-facets} applies to every non-degenerate symmetric
probability measure with mass at the endpoint.  In particular, it applies to every
symmetric bounded discrete log-concave probability measure.
\end{corollary}

\begin{proof}
For all sufficiently small $s>0$, we have that \[\mu(\{x^\ast\})\leq\mu([x^\ast-s,x^\ast]) \leq 1.\]  Thus \eqref{eq:intro-tail} holds with
$\kappa=0$.
\end{proof}

\begin{remark}[Infinite atomic examples]\label{rem:atomic-examples}\rm 
The product theorem also contains purely atomic measures with no atom at the
endpoint.  For example, if $0<q,r<1$, then
\[
 \mu_{q,r}
 =\frac{1-r}{2}\sum_{k=1}^\infty r^{k-1}
 \bigl(\delta_{1-q^k}+\delta_{-(1-q^k)}\bigr)
\]
satisfies \eqref{eq:intro-tail} with
\[
 \kappa=\frac{\log r}{\log q}>0.
\]
More generally, if $x_k\uparrow x^*$, the gaps $x^*-x_k$ have bounded
multiplicative ratios, and
\[
 \sum_{j\geq k}p_j\asymp(x^*-x_k)^\kappa,
\]
then the symmetric atomic measure assigning mass $p_k/2$ to each of
$\pm x_k$ satisfies the endpoint assumption.
\end{remark}

\section{Why some structural assumption is needed}

The preceding theorem uses product structure, while the next one uses
log-concavity.  Without either type of structure, a many-facets conclusion of
the same order is false, even for symmetric isotropic probability measures
with full support.  The following elementary example will also clarify the
role of log-concavity in the next section.

\begin{proposition}\label{prop:non-logconcave-example}
Fix $\alpha>0$ and let $N=\lfloor e^{\alpha n}\rfloor$.  There exists a
symmetric isotropic probability measure $\nu_n$ on $\R^n$ with
$\operatorname{supp}\nu_n=\R^n$, which is not log-concave, such that for
independent $X_1,\ldots,X_N$ with law $\nu_n$ and
\[
 P_N=\conv\{X_1,\ldots,X_N\},
\]
one has
\[
 \E f_{n-1}(P_N)=(1+o(1))2^n.
\]
In particular,
\[
 \E f_{n-1}(P_N)=e^{O(n)}
 \ll n^{n/2}e^{-Cn}
\]
for every fixed $C>0$.
\end{proposition}

\begin{proof}
Let
\[
 \nu_n^0=\frac1{2n}\sum_{j=1}^n
 \bigl(\delta_{\sqrt n e_j}+\delta_{-\sqrt n e_j}\bigr).
\]
Then $\nu_n^0$ is symmetric and isotropic.  A sample from $\nu_n^0$ chooses
one of the $2n$ points $\pm\sqrt n e_j$ uniformly.  Hence
\[
 \Prob(\text{some point }\pm\sqrt n e_j\text{ is missing})
 \leq2n\left(1-\frac1{2n}\right)^N
 \leq2n e^{-N/(2n)}.
\]
If all $2n$ points occur among the samples, then
\[
 P_N=\conv\{\pm\sqrt n e_1,\ldots,\pm\sqrt n e_n\}
 =\sqrt n B_1^n,
\]
and therefore $f_{n-1}(P_N)=2^n$.

To obtain full support, let $\gamma_n$ be the standard Gaussian measure on
$\R^n$, put $\varepsilon_n=e^{-n^3}$, and define
\[
 \nu_n=(1-\varepsilon_n)\nu_n^0+\varepsilon_n\gamma_n.
\]
Both measures in the mixture are symmetric and isotropic, so the same is true
of $\nu_n$, while the Gaussian component gives
$\operatorname{supp}\nu_n=\R^n$.  The measure $\nu_n$ is not log-concave,
since it has atoms and at the same time full-dimensional support.

The probability that at least one of the $N$ samples comes from the Gaussian
component is at most $N\varepsilon_n$.  On the event that no Gaussian sample
occurs and every point $\pm\sqrt n e_j$ occurs, the polytope is
$\sqrt nB_1^n$.  On the complementary event we use the elementary bound
\[
 f_{n-1}(P_N)\leq\binom{N}{n}.
\]
Consequently,
\[
 \left|\E f_{n-1}(P_N)-2^n\right|
 \leq
 2\left(2n e^{-N/(2n)}+Ne^{-n^3}\right)\binom{N}{n}
 =o(2^n),
\]
because
\[
 \binom{N}{n}\leq\left(\frac{eN}{n}\right)^n=e^{O_\alpha(n^2)}.
\]
\end{proof}

\section{General log-concave measures}

We now turn to arbitrary log-concave probability measures.  No product or
symmetry assumption will be used.

\subsection{A random simplex estimate}

\begin{lemma}\label{lem:logconcave-simplex}
Let $\nu$ be a full-dimensional log-concave probability measure on $\R^d$ with
density $g$, and let $Z_0,\ldots,Z_d$ be independent random vectors with distribution $\nu$.
Then
\begin{equation}\label{eq:logconcave-simplex}
 \frac{c^d}{\sqrt{d!}\,\norm{g}_\infty} \leq\E\,\vol_d\bigl(\conv\{Z_0,\ldots,Z_d\}\bigr)
 \leq\frac{C^d}{\sqrt{d!}\,\norm{g}_\infty},
\end{equation}
where $c,C>0$ are absolute constants.
\end{lemma}

\begin{proof}
We will first assume that $\nu$ is isotropic. Set
\[
    D=\det(Z_1-Z_0,\ldots,Z_d-Z_0).
\]
Adjoining a first row of ones gives
\[
    D
    =\det
    \begin{pmatrix}
       1&\cdots&1\\
       Z_0&\cdots&Z_d
    \end{pmatrix}.
\]
The columns $(1,Z_j)^T$ are independent and have second-moment matrix
$I_{d+1}$.  Expanding the square of the determinant, the
independence of the columns and the fact that the second-moment matrix is the
identity imply
\[
    \E D^2=(d+1)!.
\]
Since
\[
    \vol_d\bigl(\conv\{Z_0,\ldots,Z_d\}\bigr)=\frac{|D|}{d!},
\]
on the one hand Cauchy--Schwarz yields
\[
    \E\vol_d \bigl(\conv\{Z_0,\ldots,Z_d\}\bigr) 
    \leq \frac{\sqrt{\E D^2}}{d!}
    =\frac{\sqrt{d+1}}{\sqrt{d!}}\leq \frac{C^d}{\sqrt{d!}}
\] and on the other hand the fact that $D$ is a polynomial of degree $d$ of the log-concave random vector $(Z_0,\ldots,Z_d)$ in $\mathbb{R}^{d(d+1)}$ gives \[\E\vol_d \bigl(\conv\{Z_0,\ldots,Z_d\}\bigr) 
    \geq c^{d} \frac{\sqrt{\E D^2}}{d!} \geq \frac{c^d}{\sqrt{d!}},\] by reverse H\"older inequality (see \cite[Theorem~2]{Bobkov2000}).

    Now if $m$ is the barycenter of $\nu$ and $\Sigma$ its covariance matrix and $Y_i=\Sigma^{-1/2}(Z_i-m)$ for every $i=0,\ldots,d$, the random vectors $Y_i$ are independent, log-concave and isotropic. Moreover, \[\vol_d \bigl(\conv\{Y_0,\ldots,Y_d\}\bigr)=\frac{1}{\sqrt{{\rm det }\,\Sigma}}\vol_d \bigl(\conv\{Z_0,\ldots,Z_d\}\bigr).\] To conclude, the universal lower bound for the isotropic constant, as well as the proof of the slicing conjecture, yield \[c_1 \leq \|g\|_\infty^{1/d}( {\rm det}\Sigma)^{1/(2d)} \leq c_2,\] for some absolute constants $c_1,c_2>0$.
\end{proof}

\subsection{Cram\'er levels in the space of tilt parameters}

Let $\mu$ be a centered full-dimensional log-concave probability measure on
$\R^n$ with density $f$.  We write
\[
 \Lambda_\mu(\theta)=\log\int_{\R^n}e^{\ip{\theta}{x}}f(x)\,dx
\]
and denote its Legendre transform by $\Lambda_\mu^*$.  The result is affine
invariant, so we shall work in isotropic position.

We first argue under the additional assumption that $f$ is smooth, strictly
log-concave and that $\Lambda_\mu$ is finite and smooth on all of $\R^n$.  The
general case follows at the end by approximation.

For $\theta \in \mathbb{R}$ we consider the exponential tilt \[d\mu_\theta(z)=e^{\ip{\theta}{z}-\Lambda_\mu(\theta)}d\mu(z),\] which is also log-concave with \[{\rm bar}(\mu_\theta)=\nabla \Lambda_\mu(\theta), \qquad \Cov(\mu_\theta)=\nabla^2\Lambda_\mu(\theta). \]

For $t>0$ set
\[
 \Theta_t=
 \left\{\theta\in\R^n:
 \Lambda_\mu^*(\nabla\Lambda_\mu(\theta))\leq t
 \right\}
\]
Observe that $\Theta_t$ is a star-shaped body and define
\[
 M(t)=\vol_n(\Theta_t).
\]
For $u\in S^{n-1}$, let $s_t(u)>0$ be determined by
\begin{equation}\label{eq:st-definition}
 x_t(u)=\nabla\Lambda_\mu(s_t(u)u),
 \qquad
 \Lambda_\mu^*(x_t(u))=t.
\end{equation}
Thus $s_t$ is the radial function of $\Theta_t$, and if $\sigma$ denotes
the normalized spherical probability measure,
\begin{equation}\label{eq:M-radial}
 M(t)=\omega_n\int_{S^{n-1}}s_t(u)^n\,d\sigma(u),
 \qquad
 \omega_n=\vol_n(B_2^n).
\end{equation}
Define
\begin{equation}\label{eq:lambda-definition}
 \lambda_t(u)^2
 =s_t(u)^2
 \ip{\nabla^2\Lambda_\mu(s_t(u)u)u}{u}.
\end{equation}

\begin{lemma}\label{lem:M-bounds}
There is an absolute constant $C>0$ such that
\[
 M(n)\geq e^{-Cn},
 \qquad
 M(3n/2)\leq e^{Cn}.
\]
\end{lemma}

\begin{proof}
For the lower bound, put
\[
 \mathcal L_n=
 \{\theta:\Lambda_\mu(\theta)\leq n,\ \Lambda_\mu(-\theta)\leq n\}.
\]
The standard comparison between logarithmic Laplace sublevel sets and
$L_n$-centroid bodies \cite[Lemma~2.3]{KlartagMilman}, together with the volume estimates
for centroid bodies and the reverse Santal\'o inequality
\cite{BourgainMilman}, gives
\begin{equation}\label{eq:Ln-volume}
 \vol_n(\mathcal L_n)\geq e^{-Cn}.
\end{equation}
We claim that $\frac12\mathcal L_n\subset\Theta_n$.  Indeed, if
$2\theta\in\mathcal L_n$ and $g(s)=\Lambda_\mu(s\theta)$, then $g\geq0$ because
$\mu$ is centered, and convexity gives
\[
 g'(1)\leq g(2)-g(1).
\]
Therefore
\[
 \Lambda_\mu^*(\nabla\Lambda_\mu(\theta))
 =g'(1)-g(1)
 \leq g(2)-2g(1)
 \leq n.
\]
This and \eqref{eq:Ln-volume} imply $M(n)\geq e^{-Cn}$.

For the upper bound, set
\[
 B_t=\{x:\Lambda_\mu^*(x)\leq t\}.
\]
The comparison of Cram\'er sublevel sets with one-sided centroid bodies and
the corresponding volume estimates imply
\begin{equation}\label{eq:B-volume}
 \vol_n(B_{3n/2})\leq e^{Cn};
\end{equation}
see, for example, \cite[Proposition~2.7]{GT2026}.

For $x\in B_{3n/2}$ let $\theta_x=\nabla\Lambda_\mu^*(x)$ and let $g_x$ be the
density of the exponential tilt $\mu_{\theta_x}$. By duality, its barycenter is $x$ and its covariance is
\[
 \Cov(g_x)=\nabla^2\Lambda_\mu(\theta_x)
 =\bigl(\nabla^2\Lambda_\mu^*(x)\bigr)^{-1}.
\]
By Fradelizi's inequality \cite[Theorem~4]{Fradelizi},
\[
 \norm{g_x}_\infty
 \leq e^n g_x(x)
 =e^{n+\Lambda_\mu^*(x)}f(x)
 \leq e^{5n/2}\norm{f}_\infty.
\]
The functional form of the slicing theorem, obtained from Ball's reduction
\cite{Ball} and the resolution of the slicing problem by Klartag and Lehec
\cite{KlartagLehec}, gives
\[
 \norm{f}_\infty^{1/n}\leq C
\]
for isotropic $\mu$.  Hence $\norm{g_x}_\infty\leq e^{Cn}$.

If $H_x=\nabla^2\Lambda_\mu^*(x)$, the lower bound of the isotropic constant used in the proof of
Lemma~\ref{lem:logconcave-simplex} gives
\[
 \norm{g_x}_\infty\sqrt{\det\Cov(g_x)}\geq c^n,
\]
that is,
\[
 \sqrt{\det H_x}\leq e^{Cn},
 \qquad x\in B_{3n/2}.
\]
The change of variables $x=\nabla\Lambda_\mu(\theta)$ gives
\[
 M(3n/2)
 =\int_{B_{3n/2}}\det\nabla^2\Lambda_\mu^*(x)\,dx
 \leq e^{Cn}\vol_n(B_{3n/2}),
\]
and \eqref{eq:B-volume} completes the proof.
\end{proof}

\begin{lemma}\label{lem:lambda-upper}
There is an absolute constant $C>0$ such that
\[
 \lambda_t(u)\leq C(1+t)
\]
for all $t>0$ and $u\in S^{n-1}$.
\end{lemma}

\begin{proof}
Write $s=s_t(u)$, $x=x_t(u)$, $h=\ip{x}{u}$, and consider the exponential
tilt with parameter $su$.  If $X$ is a random vector distributed according to $\mu_{su}$, under this tilt, the random variable
\[
 Y=s(\ip{X}{u}-h)
\]
is centered, log-concave, and has variance $\lambda_t(u)^2$.  Moreover,
\[
 \E_{su}e^{-Y}=e^t.
\]
Write $Y=\lambda_t(u)Z$, where $Z$ is centered, log-concave and has variance
one.  Standard one-dimensional log-concave estimates give absolute constants
$c_0,c_1>0$ such that the density of $Z$ is at least $c_1$ on
$[-c_0,-c_0/2]$.  Hence
\[
 e^t\geq c_1\int_{-c_0}^{-c_0/2}e^{-\lambda_t(u)z}\,dz
 \geq c_1 e^{c_0\lambda_t(u)}
.\]
The assertion
follows.
\end{proof}

Differentiating
\[
 \Lambda_\mu^*(\nabla\Lambda_\mu(su))
 =s\ip{\nabla\Lambda_\mu(su)}{u}-\Lambda_\mu(su)
\]
with respect to $s$, we get
\[
 \frac{dt}{ds}
 =s\ip{\nabla^2\Lambda_\mu(su)u}{u}
 =\frac{\lambda_t(u)^2}{s}.
\]
Thus \eqref{eq:M-radial} implies
\begin{equation}\label{eq:M-prime-radial}
 M'(t)
 =n\omega_n\int_{S^{n-1}}
 \frac{s_t(u)^n}{\lambda_t(u)^2}\,d\sigma(u).
\end{equation}

\begin{proposition}\label{prop:direction-set}
There are an absolute $C>0$, a level
\[
 t\in[n,3n/2],
\]
a number
\[
 c\sqrt n\leq a\leq Cn,
\]
and a measurable set $S\subset S^{n-1}$ such that
\[
 a\leq\lambda_t(u)<2a,
 \qquad u\in S,
\]
and
\begin{equation}\label{eq:st-mass}
 \int_Ss_t(u)^n\,d\sigma(u)
 \geq\frac{e^{-Cn}}{\omega_n}.
\end{equation}
\end{proposition}

\begin{proof}
Lemma~\ref{lem:M-bounds} gives
\[
 \int_n^{3n/2}\frac{M'(r)}{M(r)}\,dr
 =\log\frac{M(3n/2)}{M(n)}
 \leq Cn.
\]
Hence there is $t\in[n,3n/2]$ with
\[
 M'(t)\leq C_0M(t)
\]
for an absolute $C_0$.  Let
\[
 R=\left\{u:\lambda_t(u)^2\geq\frac{n}{2C_0}\right\}.
\]
By \eqref{eq:M-prime-radial},
\[
 \int_{S^{n-1}\setminus R}s_t(u)^n\,d\sigma(u)
 \leq\frac{M(t)}{2\omega_n}.
\]
Together with \eqref{eq:M-radial},
\[
 \int_Rs_t(u)^n\,d\sigma(u)
 \geq\frac{M(t)}{2\omega_n}
 \geq\frac{e^{-Cn}}{\omega_n},
\]
since $M(t)$ is increasing. On $R$, Lemma~\ref{lem:lambda-upper} gives
\[
 c\sqrt n\leq\lambda_t(u)\leq Cn.
\]
A dyadic decomposition of this interval has $O(\log n)$ pieces.  One of them,
$[a,2a)$, satisfies \eqref{eq:st-mass}, after changing the absolute constant
in the exponential.
\end{proof}

\subsection{Lower bound for the expected facets}

For $u\in S$ from Proposition~\ref{prop:direction-set}, write
\[
 s=s_t(u),
 \qquad
 x=x_t(u),
 \qquad
 h=\ip{x}{u},
 \qquad
 \lambda=\lambda_t(u).
\]
If $X$ is distributed according to $\mu_{su}$, let $g$ be the density of
\[
 Y=s(\ip{X}{u}-h).
\]
Then $Y$ is centered and log-concave with variance $\lambda^2$.  Since
$\lambda\geq c\sqrt n$, the standard one-dimensional estimates give absolute
constants $c_0,c_1,C_1>0$ such that
\begin{equation}\label{eq:g-section-bounds}
 g(y)\geq\frac{c_1}{\lambda}
 \quad(0\leq y\leq c_0),
 \qquad
 \norm{g}_\infty\leq\frac{C_1}{\lambda}.
\end{equation}
For $u \in S^{n-1}$ and $r \in \mathbb{R}$ we set \[H(u,r)=\{z: \ip{z}{u}=r\}, \qquad H^+(u,r)=\{z: \ip{z}{u} \geq r\}.\] With this notation, for $0\leq y\leq c_0$, let also
\[
 H_y=H\left(u, h+\frac ys\right),
 \qquad
 H_y^+=H^+\left(u, h+\frac ys\right).
\]
If
\[
 p(u,r)=\int_{H(u,r)}f(z)\,d\cH^{n-1}(z),
\]
then the definition of the tilt gives
\begin{equation}\label{eq:section-mass-logconcave}
 g(y)=\frac1s e^{t+y}
 p\left(u,h+\frac ys\right).
\end{equation}
Consequently,
\begin{equation}\label{eq:section-mass-lower}
 p\left(u,h+\frac ys\right)
 \geq c\frac{s e^{-t}}\lambda,
 \qquad 0\leq y\leq c_0.
\end{equation}
Also,
\begin{equation}\label{eq:halfspace-upper-logconcave}
 \begin{aligned}
 \mu(H_y^+)
 &=e^{-t}\E_{su}\bigl[e^{-Y}\1_{\{Y\geq y\}}\bigr]\\
 &\leq C\frac{e^{-t}}\lambda
 \leq C\frac{e^{-t}}a.
 \end{aligned}
\end{equation}

\begin{theorem}\label{thm:general-logconcave}
There exists an absolute constant $C>0$ such that the following holds.  For
every sufficiently large $n$ and every full-dimensional log-concave
probability measure $\mu$ on $\R^n$, there exist $T\in[n,2n]$ and
\[
 N=\left\lceil e^Tn^{3/2}\right\rceil
\]
such that, if $X_1,\ldots,X_N$ are independent with distribution $\mu$ and
\[
 P_N=\conv\{X_1,\ldots,X_N\},
\]
then
\[
 \E f_{n-1}(P_N)
 \geq n^{n/2}e^{-Cn}
 \geq(cn)^{n/2}.
\]
\end{theorem}

\begin{proof}
By affine invariance we work in isotropic position and use the smooth setting
described above.  Choose $t,a,S$ as in
Proposition~\ref{prop:direction-set}.  Fix a sufficiently large absolute
constant $A>0$ and set
\begin{equation}\label{eq:N-logconcave}
 N=\left\lceil Ae^tna\right\rceil.
\end{equation}
For $u\in S$ and $0\leq y\leq c_0$,
\eqref{eq:halfspace-upper-logconcave} gives
\[
 N\mu(H_y^+)\leq C n.
\]
Since $t\geq n$, $\mu(H_y^+)=o(1)$, and therefore
\begin{equation}\label{eq:empty-side-logconcave}
 \bigl(1-\mu(H_y^+)\bigr)^{N-n}
 \geq e^{-C n}.
\end{equation}

For $u\in S^{n-1}$ and $r>0$ define
\[
 J(u,r)=
 \int_{H(u,r)^n}
 \cH^{n-1}\bigl(\conv\{z_1,\ldots,z_n\}\bigr)
 \prod_{i=1}^nf(z_i)\,d\cH^{n-1}(z_i).
\] Observe now that \[J(u,r)=p(u,r)^n\,\E\,\cH^{n-1}\bigl(\conv\{Z_1,\ldots,Z_n\}\bigr),\] where $Z_1,\ldots,Z_n$ are i.i.d $(n-1)$-dimensional log-concave random vectors, with density \[\frac{f(z)}{p(u,r)} \mathbf{1}_{H(u,r)}\]

Combining Lemma~\ref{lem:logconcave-simplex} with \eqref{eq:section-mass-lower},
\begin{equation}\label{eq:J-lower-logconcave}
 J\left(u,h+\frac ys\right)
 \geq
 \frac{e^{-Cn}e^{-tn}}{\sqrt{(n-1)!}}
 \frac{s^{n+1}}{\lambda^{n+1}},
 \qquad 0\leq y\leq c_0.
\end{equation}

The affine Blaschke--Petkantschin formula \cite{SchneiderWeil} gives the
following lower bound for the expected number of facets:
\[
 \E f_{n-1}(P_N)
 \geq\binom Nn I,
\]
where
\[
 I=n!\omega_n
 \int_{S^{n-1}}\int_0^\infty
 \bigl(1-\mu(H^+(u,r))\bigr)^{N-n}
 J(u,r)\,dr\,d\sigma(u).
\]
We restrict the integral to $u\in S$ and
$r=h+y/s$, $0\leq y\leq c_0$.  Since $dr=dy/s$, equations
\eqref{eq:empty-side-logconcave} and \eqref{eq:J-lower-logconcave} give
\[
 \begin{aligned}
 I
 &\geq
 \frac{e^{-Cn}n!\omega_ne^{-tn}}{\sqrt{(n-1)!}}
 \int_S\frac{s_t(u)^n}{\lambda_t(u)^{n+1}}\,d\sigma(u)\\
 &\geq
 \frac{e^{-Cn}n!\omega_ne^{-tn}}
 {\sqrt{(n-1)!}\,a^{n+1}}
 \int_Ss_t(u)^n\,d\sigma(u).
 \end{aligned}
\]
By \eqref{eq:st-mass},
\begin{equation}\label{eq:I-final-logconcave}
 I\geq
 \frac{e^{-Cn}n!e^{-tn}}
 {\sqrt{(n-1)!}\,a^{n+1}}.
\end{equation}
Since $N\gg n$, we have
$\binom Nn n!\geq(N/2)^n$.  Using
\eqref{eq:N-logconcave} and \eqref{eq:I-final-logconcave},
\[
 \E f_{n-1}(P_N)
 \geq
 e^{-Cn}\frac{n^n}{a\sqrt{(n-1)!}}.
\]
Because $a\leq Cn$, Stirling's formula gives
\[
 \E f_{n-1}(P_N)
 \geq n^{n/2}e^{-Cn}.
\]

Finally, from
\[
 n\leq t\leq\frac{3n}{2},
 \qquad
 c\sqrt n\leq a\leq Cn,
\]
and the choice of the absolute constant $A$, we may write
\[
 N=\left\lceil e^Tn^{3/2}\right\rceil
\]
with $T\in[n,2n]$ for all sufficiently large $n$.

For a general full-dimensional log-concave measure, truncate to large balls,
regularize by a small Gaussian convolution, recenter, and apply an invertible
linear map.  The resulting measures converge in total variation to $\mu$, and
for fixed $N$ the facet-count functional is continuous outside a set of
$\mu^N$-measure zero.  Since for fixed $n$ only finitely many integer values of
$N$ occur in the range
$[e^nn^{3/2},e^{2n}n^{3/2}+1]$, a subsequence has constant $N$; passing to the
limit proves the theorem without the temporary regularity assumptions.
\end{proof}

\subsection{Upper bound for the expected facets}

\begin{theorem}\label{thm:upper}
There exists an absolute constant $C>0$ with the following property. For
$n\geq2$, let $\mu$ be a full-dimensional log-concave probability measure on
$\R^n$ with density $f$, and  $X_1,\ldots,X_N$ be independent random
vectors with law $\mu$, where $N\geq2n$.  If
\[
    P_N=\conv\{X_1,\ldots,X_N\},
\]
then
\begin{equation}\label{eq:main-upper}
    \E f_{n-1}(P_N)
    \leq
    C^n N n^{(n-3)/2}.
\end{equation}
In particular, for every fixed $b>0$, if
\[
    N\leq e^{bn}n^{3/2},
\]
then
\begin{equation}\label{eq:exp-window-upper}
    \E f_{n-1}(P_N)
    \leq C_b^n n^{n/2}.
\end{equation}
\end{theorem}

\begin{proof}
By affine invariance we may assume that $\mu$ is isotropic. As in the lower-bound argument, we first work under the same
regularity assumptions on $\Lambda_\mu$; the general case is recovered by the
same approximation at the end.

For $u\in S^{n-1}$ we write \[\pi(u,r):=\mu(H^+(u,r)).\]
Blaschke--Petkantschin yields \begin{equation}\label{eq:blaschke-formula}\E f_{n-1}(P_N)=\binom{N}{n}I,\end{equation} where \[I:=n!\omega_n \int_{S^{n-1}}\int_0^{\infty} \bigl((1-\pi(u,r))^{N-n}+\pi(u,r)^{N-n}\bigr) J(u,r) \,dr\, d\sigma(u). \]  
Since $\mu$ is centered, Gr\"unbaum's lemma gives \[\pi(u,r) \leq \mu(\{z: \langle z,u\rangle \geq 0\}) \leq 1-e^{-1}\] for every $u \in S^{n-1}$ and every $r>0$. Hence, \begin{equation}\label{eq:first-integral-bound}
\begin{aligned}
I &\leq e^{-c_1(N-n)}+n!\omega_n \int_{S^{n-1}}\int_0^{\infty} (1-\pi(u,r))^{N-n} J(u,r) \,dr\, d\sigma(u)\\
&:=e^{-c_1(N-n)}+I_1.
\end{aligned}
\end{equation}

For every $r>0$ and every $u \in S^{n-1}$, there is a unique $s=s(u,r)>0$ such that
\begin{equation}\label{eq:r-s-relation}
    r=\ip{\nabla\Lambda_\mu(su)}{u}.
\end{equation}
For this particular pair $(u,r)$ set
\begin{equation}\label{eq:local-cramer-point}
    x=x(u,r):=\nabla\Lambda_\mu(su),
    \qquad
    t=t(u,r):=\Lambda_\mu^*(x)=sr-\Lambda_\mu(su),
\end{equation}
and write
\begin{equation}\label{eq:local-lambda}
    \theta_x:=\nabla\Lambda_\mu^*(x)=su,
    \qquad
    H_x:=\nabla^2\Lambda_\mu^*(x),
    \qquad
    \lambda=\lambda(u,r):=
    \sqrt{\ip{\theta_x}{H_x^{-1}\theta_x}}.
\end{equation}
Observe that
\[
    \ip{x}{u}=r, \quad t>0.
\]
Therefore, the point $x=x(u,r)$ belongs to the specific hyperplane $H(u,r)$.
Moreover,
\begin{equation}\label{eq:lambda-ray}
    \lambda^2
    =
    s^2\ip{\nabla^2\Lambda_\mu(su)u}{u}.
\end{equation}
Let $\mu_{\theta_x}$ be the exponential tilt corresponding to $\theta_x$, and put
\[
    Z_x=\ip{\theta_x}{Y-x},
    \qquad Y\sim\mu_{\theta_x}.
\]
Then $Z_x$ is centered, log-concave, and has variance $\lambda^2$. If $g_x$
denotes its density, \eqref{eq:section-mass-logconcave} implies
\begin{equation}\label{eq:rho-at-cramer}
    g_x(0)=\frac{e^t}{s}p(u,r).
\end{equation}
Hence, \eqref{eq:g-section-bounds} yields
\begin{equation}\label{eq:rho-upper-local}
    p(u,r)
    \leq
    C\,\frac{s e^{-t}}{\lambda}.
\end{equation}
Moreover, a standard argument for centered log-concave random variables gives
\begin{equation}\label{eq:cap-lower-local}
    \pi(u,r)=
    e^{-t}\,\E_{\mu_{\theta_x}}
    \left[e^{-Z_x}\1_{\{Z_x\geq0\}}\right]
    \geq
    c\,\frac{e^{-t}}{1+\lambda},
\end{equation}
where $c>0$ is an absolute constant.
We shall also use the following two elementary consequences of the same
one-dimensional normalization:

\begin{lemma}\label{lem:lambda-t-local}
We have that 
    \[\lambda\leq C(1+t).\]
 In addition, there exist absolute constants $t_0,c_0>0$ such that \[t\geq t_0\ \Longrightarrow\ \lambda\geq c_0.\]
\end{lemma}

\begin{proof}
The first part is proven exactly as Lemma~\ref{lem:lambda-upper}.
For the second part, the standard one-dimensional
$\psi_1$ estimate gives \[\E_{\mu_{\theta_x}} e^{a|Z_x|/\lambda}\leq C,\] for some absolute constants $a,C>0$. Thus,
if $\lambda\leq a$ \[e^t=\E_{\theta_x}e^{-Z_x}\leq\E_{\theta_x}e^{a|Z_x|/\lambda} \leq C.\] The proof follows.
\end{proof}
If this Cram\'er level $t=t(u,r)$ satisfies
$t<t_0$, then Lemma~\ref{lem:lambda-t-local} and
\eqref{eq:cap-lower-local} give
\begin{equation}\label{eq:central-cap-positive}
    \pi(u,r)\geq p_0
\end{equation}
for an absolute $p_0>0$. Hence, \begin{equation}\label{eq:second-integral-bound}
\begin{aligned}I_1 &\leq e^{-c_2(N-n)}+n!\omega_n\iint_{\{(u,r):t(u,r)\geq t_0\}} (1-\pi(u,r))^{N-n} J(u,r)\, dr\,d\sigma(u)\\
&:=e^{-c_2(N-n)}+I_2. 
\end{aligned}
\end{equation}

If $u \in S^{n-1}$ and $r>0$, observe  that \[J(u,r)=p(u,r)^n\,\E\,\cH^{n-1}\bigl(\conv\{W_1,\ldots,W_n\}\bigr),\] where $W_1,\ldots,W_n$ are i.i.d $(n-1)$-dimensional log-concave random vectors, with density \[\frac{f(z)}{p(u,r)} \mathbf{1}_{H(u,r)}\]
Therefore Lemma~\ref{lem:logconcave-simplex} 
 yields,
\[
    J(u,r)
    \leq
    \frac{C^n}{\sqrt{(n-1)!}}
    \frac{p(u,r)^{n+1}}{\|f \,\mathbf{1}_{H(u,r)}\|_\infty} \leq \frac{C^n}{\sqrt{(n-1)!}}
    \frac{p(u,r)^{n+1}}{f(x)},
\]
since by \eqref{eq:local-cramer-point}, the point $x=x(u,r)$ lies on
$H(u,r)$. Using \eqref{eq:rho-upper-local}, we obtain
\begin{equation}\label{eq:J-cramer-upper}
    J(u,r)
    \leq
    \frac{C^n}{\sqrt{(n-1)!}}
    \frac{s^{n+1}e^{-(n+1)t}}
         {f(x)\lambda^{n+1}}.
\end{equation}

\begin{lemma}\label{lem:jacobian-local}
The parametrization $(u,r)\mapsto x(u,r)$ has Jacobian
\[
    dr\,d\sigma(u)
    =
    \frac{1}{n\omega_n}
    \frac{\det H_x\,\lambda^2}{s^{n+1}}\,dx\]
\end{lemma}

\begin{proof}
The change of variables in the lemma is the composition \[ (u,r)\longleftrightarrow (u,s) \longleftrightarrow \theta=su \longleftrightarrow x=\nabla\Lambda(\theta). \] 
We justify each step separately. As for fixed $u\in S^{n-1}$, we have defined
\[
    r=\ip{\nabla\Lambda_\mu(su)}{u},
\]
 and $(u,s)\mapsto (u,r_u(s))$ leaves the first coordinate unchanged
\[
    dr\,d\sigma(u)  =
    \ip{\nabla^2\Lambda_\mu(su)u}{u}\,ds \,d\sigma(u) =
    \frac{\lambda^2}{s^2}\,ds\,d\sigma(u),
\]
Next, writing $\theta=su$ and recalling that $\sigma$ is the probability
measure on $S^{n-1}$, the polar-coordinate formula gives
\[
    d\theta
    =
    n\omega_n s^{n-1}\,ds\,d\sigma(u).
\]
Finally, the map
\[
    x\longmapsto \theta_x=\nabla\Lambda^\ast_\mu(x)
\]
is a diffeomorphism. Hence,
\[
    d\theta=\det H_x\,dx.
\]
Combining the preceding identities, we obtain
\[
\begin{aligned}
    dr\,d\sigma(u)
    &=
    \frac{\lambda^2}{s^2}\,ds\,d\sigma(u)\\
    &=
    \frac{\lambda^2}{s^2}
    \frac{1}{n\omega_n s^{n-1}}\,d\theta\\
    &=
    \frac{1}{n\omega_n}
    \frac{\det H_x\,\lambda^2}{s^{n+1}}\,dx,
\end{aligned}
\]
which proves the lemma.
\end{proof}
We now need one pointwise estimate which is specific to the upper bound. Let
$g_x$ denote the density of the tilt $\mu_{\theta_x}$. Since $x$ is the barycenter of
$\mu_{\theta_x}$,
\[
    g_x(x)=e^t f(x),
    \qquad
    \Cov(\mu_{\theta_x})=H_x^{-1}.
\]
Fradelizi's inequality gives
\[
    \norm{g_x}_\infty
    \leq e^n g_x(x)
    =
    e^{n+t}f(x).
\]
On the other hand, the universal lower bound for the isotropic constant of a
log-concave density gives
\[
    c^n
    \leq
    \norm{g_x}_\infty\sqrt{\det\Cov(\mu_{\theta_x})}
    =
    \frac{\norm{g_x}_\infty}{\sqrt{\det H_x}}.
\]
Consequently,
\begin{equation}\label{eq:determinant-cancellation-local}
    \frac{\det H_x}{f(x)}
    \leq
    C^n e^{2t}f(x).
\end{equation}

From \eqref{eq:cap-lower-local} and $N\geq2n$,
\begin{equation}\label{eq:empty-cap-upper}
\begin{split}
    (1-\pi(u,r))^{N-n}
    &\leq \exp\bigl(-(N-n)\pi(u,r)\bigr)\\
    &\leq
    \exp\left(-c\frac{Ne^{-t}}{1+\lambda}\right).
\end{split}
\end{equation}
Inserting \eqref{eq:J-cramer-upper}, the Jacobian
Lemma~\ref{lem:jacobian-local}, and finally
\eqref{eq:determinant-cancellation-local} into \eqref{eq:second-integral-bound}, we obtain
\begin{align}
    I_2
    &\leq
    \frac{C^n n!}{n\sqrt{(n-1)!}}
    \int_{\{x:t(x)\geq t_0\}}
    e^{-(n+1)t}
    \frac{\det H_x}{f(x)\lambda^{n-1}}
    \exp\left(-c\frac{Ne^{-t}}{1+\lambda}\right)
    dx
    \notag\\
    &\leq
    C^n\sqrt{(n-1)!}
    \int_{\{x:t(x)\geq t_0\}}
    f(x)
    \frac{e^{-(n-1)t}}{\lambda^{n-1}}
    \exp\left(-c\frac{Ne^{-t}}{1+\lambda}\right)
    dx.
    \label{eq:I-after-cancellation}
\end{align}
  On $\{x:t(x)\geq t_0\}$,
Lemma~\ref{lem:lambda-t-local} gives $\lambda\geq c_0$.  Introduce
\begin{equation}\label{eq:y-substitution-pointwise}
    y=\frac{Ne^{-t}}{1+\lambda}.
\end{equation}
Then, pointwise, for $d=n-1$
\begin{align}
    \frac{e^{-dt}}{\lambda^d}e^{-cy}
    &=
    N^{-d}y^d
    \left(\frac{1+\lambda}{\lambda}\right)^d
    e^{-cy}
    \notag\\
    &\leq
    C^dN^{-d}y^de^{-cy}
    \notag\\
    &\leq
    C^dN^{-d}d^d,
    \label{eq:pointwise-saddle}
\end{align}
where in the last step we used
\[
    \sup_{y\geq0}y^de^{-cy}
    =\left(\frac{d}{ce}\right)^d
    \leq C^dd^d.
\]
Substituting \eqref{eq:pointwise-saddle} into
\eqref{eq:I-after-cancellation} and using $\int f=1$, we get
\begin{equation}\label{eq:I-final}
    I_2
    \leq
    C^n\sqrt{(n-1)!}
    \frac{(n-1)^{n-1}}{N^{n-1}}.
\end{equation}
Therefore, combining \eqref{eq:first-integral-bound}, \eqref{eq:second-integral-bound} and \eqref{eq:I-final}
\[
    \E f_{n-1}(P_N)
    \leq 2\binom{N}{n}e^{-c(N-n)}+
    \binom Nn C^n \sqrt{(n-1)!}\, \frac{(n-1)^{n-1}}{N^{n-1}}.
    \]
For $N\geq2n$, 
\[
   2 \binom Nn e^{-c(N-n)}
    \leq
    \left(\frac{2eN}{n}\right)^n e^{-cN/2},
\]
and the right-hand side is at most $e^{C_1n}$ uniformly for $N/n\geq2$.
This error will be absorbed into the final estimate. Consequently,
\begin{align}
    \E f_{n-1}(P_N) &\leq e^{C_1n}+ C^n
    \frac{N^n}{n!}
    \sqrt{(n-1)!}
    \frac{(n-1)^{n-1}}{N^{n-1}}
    \notag\\
    &=
    e^{C_1n}+C^nN
    \frac{(n-1)^{n-1}}{n\sqrt{(n-1)!}}
    \label{eq:before-stirling}
\end{align}
By Stirling's formula,
\begin{equation}\label{eq:stirling-last}
    \frac{(n-1)^{n-1}}{n\sqrt{(n-1)!}}
    \leq C^n n^{(n-3)/2}.
\end{equation}
After increasing $C$, the error $e^{C_1n}$ is also bounded by the right-hand
side of \eqref{eq:stirling-last}, since $N\geq2n$.  This proves
\eqref{eq:main-upper}.

Finally, if $N\leq e^{bn}n^{3/2}$, then
\[
    C^nNn^{(n-3)/2}
    \leq
    (Ce^b)^n n^{3/2}n^{(n-3)/2}
    =(Ce^b)^n n^{n/2},
\]
which proves \eqref{eq:exp-window-upper} under the regularity assumptions.

The regularity assumptions are removed exactly as in the lower-bound
argument. For fixed $N$, the approximating product measures converge
in total variation, while
\[
    f_{n-1}(P_N)\leq\binom Nn.
\]
Hence the expectations converge as well. This completes the proof of
Theorem~\ref{thm:upper}.
\end{proof}
\begin{remark}\rm 
The preceding estimate is genuinely an average-case statement and has no
pointwise analogue, even for the standard Gaussian measure.  Indeed, McMullen's Upper Bound Theorem states that among all
$n$-dimensional polytopes with $N$ vertices the cyclic polytope
$C(N,n)$ has the largest possible number of facets. Since the Gaussian
density is strictly positive everywhere, and the combinatorial type of a
fixed realization of $C(N,n)$ is preserved under sufficiently small
perturbations, a Gaussian random polytope has, with positive probability,
exactly this maximal number of facets.
\[
 f_{n-1}(C(N,n))
 =
 \begin{cases}
 \displaystyle
 \binom{N-m}{m}+\binom{N-m-1}{m-1},
 & n=2m,\\[2ex]
 \displaystyle
 2\binom{N-m-1}{m},
 & n=2m+1.
 \end{cases}
\]
Thus, for fixed $n$ and large $N$, the maximal number of facets is of order
$N^{\lfloor n/2\rfloor}$, whereas the expectation in the exponential regime
considered above is only $n^{n/2}e^{O(n)}$.  This is in sharp contrast with
discrete models such as $0/1$ polytopes, where one may hope for a
corresponding bound on the maximal number of facets.
\end{remark}

\bibliographystyle{plain}
\bibliography{many_facets_product_logconcave}

\end{document}